\documentclass[reqno]{amsart}
\usepackage{amsmath,amsfonts,amssymb,amscd,amsthm,amsbsy,upref,mathrsfs,cleveref}

\usepackage{marginnote}
\newtheorem{theorem}{Theorem}[section]
\newtheorem{lemma}[theorem]{Lemma}
\newtheorem{proposition}[theorem]{Proposition}
\newtheorem{corollary}[theorem]{Corollary}

\theoremstyle{definition}
\newtheorem{definition}[theorem]{Definition}

\newtheorem{remark}[theorem]{Remark}
\newtheorem{remarks}[theorem]{Remarks}
\newcommand{\N}{{\mathbb{N}}}

\newcommand{\al}{\alpha}
\newcommand{\be}{\beta}
\newcommand{\e}{\varepsilon}

\newcommand{\mc}{\mathcal}

\newcommand{\fa}{f_{\alpha}}
\newcommand{\fb}{f_{\beta}}
\newcommand{\ga}{g_{\alpha}}

\numberwithin{subsection}{section}
\numberwithin{equation}{section}

\newcommand{\norm}[1][\cdot]{\lVert #1\rVert}
\newcommand{\vvert}[1][\cdot]{\vert #1\vert}

\DeclareMathOperator{\supp}{supp}
\DeclareMathOperator{\maxsupp}{maxsupp}
\DeclareMathOperator{\minsupp}{minsupp}
\DeclareMathOperator{\maxran}{maxran}

\DeclareMathOperator{\ran}{range}

\begin{document}
\author{Spiros A. Argyros, Antonis Manoussakis, Anna Pelczar-Barwacz}
\address[S.A. Argyros]
{Department of Mathematics, National Technical University of Athens, Athens 15780, Greece}
\email{sargyros@math.ntua.gr}
\address[A. Manoussakis]{Department of Sciences, Technical University of Crete, GR 73100, Greece}
\email{amanousakis@isc.tuc.gr}
\address[A. Pelczar-Barwacz]{Institute of Mathematics, Jagiellonian University, {\L}ojasiewicza 6, 30-348 Krak\'ow, Poland}
\email{anna.pelczar@im.uj.edu.pl}
\title{A type $(4)$ space in (FR)-classification}
\begin{abstract}
We present a reflexive Banach space with an unconditional basis which is quasi-minimal and tight by range, i.e. of type (4) in Ferenczi-Rosendal list within the framework of Gowers' classification program of Banach spaces. The space is an unconditional variant of the Gowers  Hereditarily Indecomposable space with asymptotically unconditional basis.
\end{abstract}
\date{\today}
\maketitle
\section*{Introduction}
In the celebrated papers \cite{G3,G4} W.T. Gowers started his classification program for  Banach spaces. The goal is to identify classes of Banach spaces which are
\begin{itemize}
 \item
hereditary, i.e. if a space belongs to a given class, then all of
its closed infinite dimensional subspaces belong to the same class
as well as well,
 \item inevitable, i.e. any Banach space contains an infinite dimensional subspace in one of those classes,
\item defined in terms of richness of family of bounded operators in the space.
\end{itemize}
The famous Gowers' dichotomy brought the first two classes: spaces
with an unconditional basis and hereditary indecomposable spaces.
Recall that a space is called \textit{hereditarily indecomposable}
(HI) if none of its infinite dimensional subspaces can be written
as a direct sum of two closed infinite dimensional subspaces.

Further classes were defined in terms of the family of
isomorphisms defined in a space. Recall that a Banach space is
minimal if it embeds isomorphically into any of its closed
infinite dimensional subspaces. Relaxing of this notion on obtains
quasi-minimality, which asserts that any two subspaces of a given
space contain further two isomorphic subspace. W.T. Gowers
obtained a dichotomy between quasi-minimality and tightness by
support in \cite{G4}.  The latter notion, among other types of
tightness, was explicity defined and studied in \cite{FR}. Recall
that a subspace $Y$ of a Banach space $X$ with a basis $(e_n)$ is
\textit{tight in} $X$ iff there is a sequence of successive
subsets $I_1<I_2<\dots$ of $\N$ such that the support of any
isomorphic copy of $Y$ in $X$ intersects all but finitely many
$I_n$'s. $X$ is called \textit{tight} if any of its subspaces is
tight in $X$. Adding requirements on the subsets $(I_n)$ with
respect to the given $Y$ one obtains more specific notions, in
particular in \textit{tightness by support} mentioned above the
subsets witnessing tightness of a subspace $Y$ spanned by a block
sequence $(x_n)$ are chosen to be supports of $(x_n)$ \cite{G4}.

V.Ferenczi and C.Rosendal have presented in \cite{FR} further
dichotomies refining Gowers list of classes: the "third dichotomy"
contrasting tightness with minimality and the "forth dichotomy"
between \textit{tightness by range}, where the subsets witnessing
the tightness of a subspace $Y$ spanned by a block sequence
$(x_n)$ are chosen to be ranges of $(x_n)$, with a stronger form
of quasi-minimality, namely sequential minimality. A Banach space
$X$ is \textit{sequentially minimal} if it is quasi-minimal and is
block saturated with block sequences $(x_n)$ with the following
property: any subspace of $X$ contains a sequence equivalent to a
subsequence of $(x_n)$.

The obvious observations relate some of the properties listed
above to HI/uncon\-di\-tio\-nal dichotomy - in particular clearly
any HI space is quasi-minimal and any tight basis is
unconditional. V.Ferenczi and C.Rosendal in \cite{FR2} studied the
spaces already known identifying their properties with respect to
the dichotomies mentioned above. Their study left open two
particular cases. Namely, an HI and sequentially minimal space and
also a quasi-minimal and tight by range space with an
unconditional basis. The answer to the first question was provided
by a version of Gowers-Maurey HI space, as it was proved by
V.Ferenczi and Th.Schlumprecht recently \cite{FS}. We recall now
the list of classes developed in \cite{FR} as stated in \cite{FS},
mentioning also some already known examples.
\begin{theorem}[(FR)-classification]
Any infinite dimensional Banach space contains a subspace from one of the following classes:
\begin{enumerate}
\item HI, tight by range (Gowers space with asymptotically unconditional basis \cite{G2,FR2}),
\item HI, tight, sequentially minimal (a version of Gowers-Maurey space, \cite{FS}),
\item tight by support (Gowers space with unconditional basis \cite{G,FR}),
\item with unconditional basis, tight by range, quasi-minimal (?),
\item with unconditional basis, tight, sequentially minimal (Tsirelson space \cite{FR2}),
\item with unconditional basis, minimal ($\ell_p$, $c_0$, dual to Tsirelson space \cite{CJT}, Schlum\-precht space \cite{AS})
\end{enumerate}
\end{theorem}

The aim of the present paper is to construct a reflexive space of
type (4) in the above classification. Namely the following is
proven.

\begin{theorem}
There exists a reflexive space $\mathcal{X}_{(4)}$ with an
unconditional basis which is quasi-minimal and tight by range.
\end{theorem}

As we have mentioned the space $\mathcal{X}_{(4)}$ is the
unconditional version of Gowers HI space which is asymptotically
unconditional \cite{G}. Banach spaces with an unconditional basis
which are variants of HI spaces have occured with Gowers' solution
of the hyperplane problem \cite{G2} and were followed by the most
recent \cite{ABR,ABR2}. Among the features of those spaces
is the non homogeneous structure. For example in all \cite{G2, ABR, ABR2} the spaces are tight by support. The new
phenomenon in the present construction is that the space
$\mathcal{X}_{(4)}$ is quasi-minimal. This is a consequence of the
definition of the norming set $W$, which is slightly different
from the initial Gowers definition, in the following manner.
Starting with an appropriately chosen double sequence $(m_j,
n_j)_j$ we consider the following norming sets $W_1, W_2$.

The set $W_1$ is the smallest subset of $c_{00}(\mathbb{N})$
satisfying
\begin{itemize}

\item[(i)] $W_1$ contains $(e_n)_n$

\item[(ii)] For every $f\in W_1$ and $g\in c_{00}$ with $|g| =
|f|$, then $g\in W_1$

\item[(iii)] It is closed in the projections on the subsets of
$\mathbb{N}$

\item[(iv)] It is closed in the even operations
$(\frac{1}{m_{2j}},\mathcal{A}_{n_{2j}})$

\item[(v)] It is closed in the odd operations
$(\frac{1}{m_{2j+1}},\mathcal{A}_{n_{2j+1}})$ on special sequences
$f_1,f_2,\ldots,f_{n_{2j+1}}$ (Here $f_1,f_2,\ldots,f_{n_{2j+1}}$
is a special sequence if the weight if each $f_i$ is even and for
$1<i$ the weight of $f_i$ is uniquely determined by the sequence
$|f_1|, |f_2|,\ldots,|f_{i-1}|$)

\end{itemize}

Let $\|\cdot\|_1$ be the norm induced on $c_{00}(\mathbb{N})$ by
the set $W_1$ and $\mathfrak{X}_1$ its completion. Then the space
$\mathfrak{X}_1$ is reflexive with a 1-unconditional basis, tight
by support (hence not quasi-minimal) and shares all the properties
of Gowers space \cite{G2}.

Consider next the norming set $W_2$ which satisfies properties
(i), (ii), (iv), (v) of $W_1$ and the following
\begin{itemize}

\item[(iii)$^\prime$] The set $W_2$ is closed in the projections
of its elements on intervals of $\mathbb{N}$.

\end{itemize}

Denoting by $\|\cdot\|_2$ the norm induced by $W_2$ and
$\mathfrak{X}_2$ the corresponding completion, the space
$\mathfrak{X}_2$ is reflexive with a 1-unconditional basis and
quasi-minimal.

This key difference between $\mathfrak{X}_1$ and $\mathfrak{X}_2$
permits the construction of the space $\mathcal{X}_{(4)}$. The
norming set $W$ of the space $\mathcal{X}_{(4)}$ is the smallest
subset of $c_{00}$ satisfying all the properties of the set $W_2$
and an additional one, called ``Gowers operation'' which is used
to show that the space is tight by range.

It is unclear to us what the structure of the space of the
operators $\mathcal{L}(\mathcal{X}_{(4)})$ is. We recall from
\cite{GM2} that every bounded linear operator on $\mathfrak{X}_1$
is of the form $D+S$ with $D$ a diagonal operator and $S$ a
strictly singular one. Such a property seems to fail for the space
$\mathfrak{X}_2$.

We describe now briefly the content of the paper. The first
section is devoted to the definition of the norming set of the
space $\mc{X}_4$. The second section contains the basic
estimations, providing tools to be used in the last two sections
in order to show quasi-minimality and tightness by range of the
space $\mathcal{X}_{(4)}$.

\section{The norming set $W$}
Let us recall the usual basic notation. Let $X$ be a  Banach space with basis $(e_i)$. The \textit{support} of a vector $x=\sum_{i} x_i e_i$ is the set $\supp x =\{ i\in \N : x_i\neq 0\}$, the \textit{range} of $x$ - the minimal interval containing $\supp x$. Given any $x=\sum_{i} a_ie_i$ and finite $E\subset\N$ put $Ex=\sum_{i\in E}a_ie_i$ and $|x|=\sum_i|a_i|e_i$. We write $x<y$ for vectors $x,y\in X$, if $\max\supp x<\min \supp y$. A \textit{block sequence} is any sequence $(x_i)\subset X$ satisfying $x_{1}<x_{2}<\dots$, a \textit{block subspace} of $X$ - any closed subspace spanned by an infinite block sequence. Given a family $\mc{F}$ of finite subsets of $\N$ we say that a block sequence $(x_{i})_{i=1}^{d}$ is $\mc{F}$-\textit{admissible} if  $(\minsupp(x_{i}))_{i=1}^{d}\in\mc{F}$. By the $( \theta,\mc{F})$-\textit{operation}, for $\theta\in (0,1]$, we mean an operation which associates with any $\mc{F}$-admissible sequence $(x_1,\dots,x_d)$ the average $\theta(x_1+\dots+x_d)$.

We define the space $\mathcal{X}_{(4)}$ to be the completion of $(c_{00}(\N)$ under the norm $\norm$ given by some set $W\subset c_{00}(\N)$, described below, as the norming set. (i.e. $\norm[x]=\sup\{f(x):f\in W\}$ for $x\in c_{00}(\N)$).

To define the set  $W$ we fix two sequences of natural numbers $(m_{j})_{j}$ and $(n_{j})_{j}$ defined recursively as follows.

We set $m_{1} =2$ and $m_{j+1}=m_{j}^{5}$ and $n_{1}=4$ and $n_{j+1}=(5n_{j})^{s_{j}}$
where $s_{j}=\log_{2}(m_{j+1}^{3})$, $j\geq 1$. We also fix  a partition of $\N$ into two infinite sets $N_{1},N_{2}$.

The set $W$ is defined to be the smallest subset of $c_{00}(\N)$ satisfying the following properties

$\alpha)$  It is unconditional (i.e. for $f\in W$, $g\in c_{00}(\N)$ with $\vvert[g]=\vvert[f]$ we have $g\in W)$.

$\be)$  It contains $(\pm e_n^*)_n$, where $(e_{n}^{*})_{n}$ is  the usual basis of $c_{00}(\N)$.

$\gamma)$ It is closed on the interval projections.

$\delta)$ It is closed  under the $(\frac{1}{m_{2j}},\mc{A}_{n_{2j}})$-operations (i.e. for every $f_{1}<f_{2}<\dots<f_{n_{2j}}$ in $W$ the functional $f=\frac{1}{m_{2j}}\sum_{i=1}^{n_{2j}}f_{i}\in W$).

$\varepsilon)$  It is closed  under the $(\frac{1}{m_{2j+1}},\mc{A}_{n_{2j+1}})$-operations  on $(2j+1)$-special sequences.

A  sequence $f_{1}<f_{2}<\dots<f_{n_{2q+1}}$ in $W$ is a $(2j+1)$-special sequence if the following are satisfied
\begin{enumerate}
\item $n_{2j+1}<m_{2j_{1}}<\dots<m_{2j_{2j+1}}$,
\item  $w(f_{1})=m_{2j_{1}}$ for some $j_{1}\in N_{1}$,
\item $w(f_i)=m_{2\sigma(|f_1|,\dots,|f_{i-1}|)}$ for any $1<i\leq n_{2j+1}$.
\item For $1< i\leq 2q+1$ the sequence $(\vvert[f_1],\vvert[f_2],\dots,\vvert[f_{i-1}])$ is uniquely determined by $w(f_{i})$
\end{enumerate}
The special sequences can be defined in a similar manner as in \cite{GM},\cite{ATod}, with the use of a coding function $\sigma$.

A functional $f=\frac{1}{m_{2j+1}}\sum_{i=1}^{n_{2j+1}}f_{i}$ with $(f_1,\dots,f_{n_{2j+1}})$ an $(2j+1)$-special sequence is called a $(2j+1)$-special functional. 

$\zeta)$ It is closed under the  $G$-operation, defined as follows.\\
For any set $F=\{n_{1}<\dots<n_{2q}\}\subset\N$ which is Schreier (i.e. $2q\leq n_1$) we set
$$
S_{F}f=\chi_{\cup_{p=1}^{q}[n_{2p-1}, n_{2p})}f.
$$

The $G$-operation associates with any $f\in c_{00}$ the vector $g=\frac{1}{2}S_{F}f$, for any $F$ as above.
\begin{remarks} $(i)$ Clearly the natural basis $(e_{n})_{n}$ is $1$-unconditional in $\mc{X}_4$. Moreover, standard argument shows that $\mc{X}$ is a reflexive space.

$(ii)$ The space $\mc{X}_{(4)}$ is an unconditional  variant of W.T. Gowers, \cite{G},  HI space with an asymptotically unconditional basis. The key ingredient in Gowers construction beyond the standard ones is an operation similar to $\zeta)$.  V.Ferenczi and C.Rosendal, \cite{FR2}, have shown that the Gowers  space is tight by range.

$(iii)$ It is worth pointing out that the quasi-minimal property
of $\mc{X}_{(4)}$ is a result of the fact that the set $W$ is not
closed in rational convex combinations. Indeed if we include the
rational convex combinations in the set $W$ (even if we exclude
property $\zeta)$) we will get a space similar to Gowers space
with an unconditional basis \cite{G2} which is tight by support
and hence not quasi-minimal \cite{FR2}.
\end{remarks}
\subsection{The analysis of a norming functional} As in the previous cases of norming sets defined to be closed under certain operations every functional $f\in W$ admits a tree-analysis which in the present case is described as follows.
\begin{definition}
  \label{treeanal}
 Let $f\in W$. A family $(\fa)_{\al\in \mc{A}}$ with  $\mc{A}$ is a rooted finite tree of finite sequences of $\N$ is a tree-analysis of $f$ if the following are satisfied
 \begin{enumerate}
 \item[1)] $f=f_{0}$ where $0$ denotes the root of $f$.
\item[2)] If $\al$ is maximal element of $\mc{A}$ then $\fa=\e e_{n}^{*}$ for some $\e=1$ or $-1$ and $n\in\N$.\\
If $\al\in\mc{A}$ is not maximal, then one of the following conditions hold
\item[3)] $\fa=\frac{1}{m_{j}}\sum_{\be\in S_{\alpha}}\fb$ where
$\fa=E_{\al}\tilde{f}_{\al}$, $E_{\al}$ interval of $\N$,
$\tilde{f}_{\al}=\frac{1}{m_j}\sum_{i=1}^{n_j}f_i$,
$S_{\al}=\{(a,i):E_{\al}f_{i}\ne 0\}$ and $\fb=E_{\al}f_i$ for
$\be=(\al,i)$. In this case we set the weight $w(\fa)$ of $\fa$ to be $w(\fa)=m_{j}$.
\item[4)] $\fa=\frac{1}{2}S_{F_{\al}}\fb$ with $\be=(\al,1)$, $S_{\al}=\{\be\}$, $F_\al$ Schreier and $\ran(\fa)=\ran(\fb)$.
 \end{enumerate}
\end{definition}
Since the  $(2j+1)$-special sequences $(f_{i})_{i\leq n_{2j+1}}$ is determined by $(\vvert[f_{i}])_{i\leq n_{2j+1}}$ it is easy to see the following
\begin{lemma}
Let $f\in W$ with a tree-analysis $(\fa)_{a\in\mc{A}}$ and $g\in W$ with $\vvert[g]=\vvert[f]$.  Then  $g$ admits a tree analysis $(\ga)_{\al\in\mc{A}}$  such that $\vvert[\ga]=\vvert[\fa]$ for all $\al\in\mc{A}$. In particular if $\fa$ is a weighted functional then $\ga$ is also weighted functional and $w(\fa)=w(\ga)$. If $\fa$ is a $G$-functional, i.e. of the form $\fa=\frac{1}{2}S_{F_\al}f_\beta$ then also $g_\al=\frac{1}{2}S_{F_\al}g_\beta$.
\end{lemma}
The following follows easily.
\begin{lemma}\label{newtree}
  Let $f\in W$ with a tree-analysis $(\fa)_{\al\in\mc{A}}$. Let  also $D\subset\mc{A}$ be a set of incomparable nodes of $\mc{A}$ and for every $\al\in D$ let  $\ga\in W$ such that $\vvert[\ga]=\vvert[\fa]$. Then there exists  $g\in W$ satisfying
  \begin{enumerate}
  \item[i)] $\vvert[g]=\vvert[f]$
    \item[ii)] $g$ admits a tree analysis
      $(\tilde{g}_{\al})_{\al\in\mc{A}}$ such that for every
      $\al\in\mc{A}$,
      $\vvert[\tilde{g}_{\al}]=\vvert[\fa]$ and for every $\al\in D$, $\tilde{g}_{\al}=\ga$.
  \end{enumerate}
\end{lemma}

\section{Basic estimations}
In this section we shall give the definition of some special
vectors as well as  estimations of the functionals of the norming
set on these special vectors.  All the definitions and estimations
have appeared in a series of  papers, \cite{S,GM,ATod}, so for the
proofs we shall refer to a paper where the corresponding result
has appeared.
\subsection{Special vectors}
\begin{definition}
  \label{saverage}
A  $C-\ell_{1}^{n}$-average, $C\geq 1$, $n\in\N$, is a vector
$x=\frac{x_{1}+\dots+x_{n}}{n}$ where $\norm[x_{i}]\leq C$,
$\norm[x]> 1$ and $n\leq x_{1}<x_{2}<\dots<x_{n}$.
\end{definition}
\begin{lemma}\label{esaverage}
Let $(x_{k})_{k}$ be a normalized block sequence. Then for every
$n\in\N$ there exists  $l(n)\in N$  such that for every finite
subsequence $(x_n)_{n\in F}$ with $\#F\geq l(n)$ of $(x_{k})_{k}$
there exists a block sequence $y_{1}<y_{2}<\dots<y_{n}$ of
$(x_{k})_{k\in F}$ such that
$\frac{y_{1}+\dots+y_{n}}{n}=\sum_{n\in F}a_{n}x_{n}$ is an
$2-\ell_{1}^{n}$-average.
\end{lemma}
The proof of the above lemma originates from \cite{S,GM}. For a proof we refer to \cite{ATod}, Lemma II.22.
\begin{definition}
  \label{dRIS}
A block sequence $(x_{k})_{k}$ is said to be a $(C,\e)$ rapidly
increasing sequence (RIS) if $\norm[x_{k}]\leq  C$ for each $k$
and there exists a strictly increasing sequence $(j_k)$   of
positive integers such that
  \begin{enumerate}
  \item $\max\ran(x_{k}) m_{j_{k+1}}^{-1}<\e$,
\item for every $k=1,2,\dots$ and every $f\in W$ with
$w(f)=m_{i}<m_{j_{k}}$ the following holds, $\vvert[f(x_{k})]\leq
\frac{C}{m_{i}}$.
\end{enumerate}
\end{definition}

\begin{definition}\label{exactpair}
  A pair  $(x,\phi)$ with $x\in \mc{X}_{(4)}$ and $\phi\in W$ is said to be a $(C,2j)$-exact pair, where $C\geq 1$, $j\in\N$, if the following conditions holds
  \begin{enumerate}
  \item $1\leq\norm[x]\leq C$, for every $\psi\in W$ with $w(\psi)<m_{2j}$ we have $\vvert[\psi(x)]\leq \frac{3C}{w(\psi)}$, while for $\psi\in W$ with $w(\psi)>m_{2j}$, $\vvert[\psi(x)]\leq \frac{C}{m_{2j}^{2}}$,
  \item $w(\phi)=m_{2j}$,
  \item $\phi(x)=1$ and $\ran(x)=\ran(\phi)$.
  \end{enumerate}
\end{definition}

\begin{definition}
  \label{4j+1-special}
We shall call the sequence   $(x_{i},x_{i}^{*})_{i=1}^{n_{2j+1}}$ a $(C,2j+1)$-dependent sequence  if
\begin{enumerate}
\item $j_{1}\in N_{1}$ and  $m_{2j_{1}}\geq n_{2j+1}$,
\item for every $i\leq  n_{2j+2}$,  $(x_{i},x_{i}^{*})$ is an  $(C, 2j_{i})$-exact pair,
\item $ (x_{1}^{*},\dots,x_{n_{2j+1}}^{*})$-is a special sequence.
\end{enumerate}
\end{definition}
\subsection{Basic estimations}
\begin{lemma}
  \label{l1average}
  Let  $x$ be a $2-\ell_{1}^{n_{j}}$-average. Then for every $f\in W$ with $w(f)=m_{i}<m_{j_{k}}$  the following holds
 \begin{equation}
   \label{eq:19}
 \vvert[f(x_{k})]\leq 3\frac{1}{m_{i}}.
 \end{equation}
\end{lemma}
We refer to \cite{S},\cite{GM} Lemma 5, \cite{ATod} Lemma II.23 for the proof.

The following result follows from Lemmas~\ref{l1average} and \ref{esaverage}
\begin{proposition}\label{ris}
For every $\e>0$ and every  block subspace $Z$ of $\mc{X}_{(4)}$ there exists a $(3,\e)$-RIS  $(x_{k})_{k}$ in $Z$.
\end{proposition}
The following proposition will be the main tool for the
estimations we shall need in the sequel. For the proof we refer to
\cite{ATod}, Propositions II.14, II.19.
\begin{proposition}
  \label{eRIS} Let $(x_{k})_{k=1}^{n_{j}}$ be a $(C,\e)$-RIS with $\e\leq  m_{j}^{-2}$ and $f\in W$. Then
  \begin{equation}
    \label{eq:7}
    \vvert[f(\frac{m_{j}}{n_{j}}\sum_{k=1}^{n_{j}}x_{k})]\leq
    \begin{cases}
      3Cw(f)^{-1}\,\,&\textrm{if}\,\,\,\,\, w(f)<m_{j}\\
       C(w(f)^{-1}m_{j}+\frac{m_{j}}{n_{j}}+m_{j}\e)\,\,&\textrm{if}\,\,\,\,
       w(f)\geq  m_{j}
    \end{cases}
  \end{equation}
In particular
$\norm[\frac{m_{j}}{n_{{j}}}\sum_{k=1}^{n_{j}}x_{k}]\leq 2C$.

Moreover if $(f_\alpha)_{\alpha\in\mathcal{A}}$ is a tree analysis
of $f$ and for every $\alpha\in\mathcal{A}$ with
$w(f_\alpha)=m_{j}$ and every interval $E$ of positive integers we
have that
$$
\vvert[f_\alpha(\sum_{k\in E}x_{k})]\leq C(1+\e\# E )
$$
then $\vvert[f(\frac{m_{j}}{n_{j}}\sum_{k=1}^{n_{j}}y_{k})]\leq \frac{4C}{m_{j}}$.
\end{proposition}
Proposition~\ref{eRIS} yields the following
\begin{proposition}
\label{exexact} For every block subspace $Z$, every $\e>0$ and $j\in\N$ there exists a $(6,2j)$-exact pair $(x,\phi)$ with $x\in
Z$.
\end{proposition}
\begin{proof}
 From Proposition~\ref{ris}  there exists $(x_k)_{k=1}^{n_{2j}}$ $(3,\e)$ RIS with $\e<1/m_{2j}^{5}$. Choose $x_{k}^{*}\in W$ with $x_{k}^{*}(x_{k})=1$ and $\ran(x_{k}^{*})=\ran(x_{k})$. Then Proposition~\ref{eRIS} yields that
 $$
(\frac{m_{j}}{n_{j}}\sum_{k=1}^{n_{2j}}x_{k},\frac{1}{m_{2j}}\sum_{i=1}^{n_{2j}}x_{k}^{*})
$$
is an $(6,2j)$-exact pair.
\end{proof}
\begin{corollary}\label{spRIS}
Let $(x_{i},x_{i}^{*})_{i=1}^{n_{2j+1}}$ be a $(6,2j+1)$-dependent sequence and  $f=\frac{1}{m_{2j+1}}E\sum_{r=1}^{n_{2k+1}}f_{r}$ a special functional  such that $w(f_{r})\ne w(x_{i}^{*})$ for every $i,r\leq n_{2j+1}$. Then
  \begin{equation}
    \label{eq:4}
    \vvert[f(\sum_{i=1}^{n_{2j+1}}x_{i})]
\leq\frac{1}{m_{2j+1}m_{2j+2}^{2} }.
  \end{equation}
\end{corollary}
 \begin{proof}
   For every $i\leq n_{2j+1}$ set
  $$
R_{i,1}=\{r\leq n_{2j+1}: \ran (f_{r})\cap  \ran(x_{i})\ne\emptyset\, \textrm{and}\, w(f_{r})<m_{2j_{i}}\}
  $$
 and
 $$R_{i,2}=\{r\leq n_{2j+1}: \ran (f_{r})\cap  \ran(x_{i})\ne\emptyset\,
 \textrm{and}\, w(f_{r})>m_{2j_{i}}\} .
 $$
Note that for every $r$ there exists at most two $i$'s such that $r\in R_{i,1}$ and $\ran(x_{i})\subsetneqq \ran(f_{r})$. From \eqref{eq:7} we get
 \begin{equation}
   \label{eq:5}
   \vvert[\sum_{r\in R_{i,1}}f_{r}(x_{i})]
   \leq
   \sum_{r\in R_{i,1}}\frac{18}{w(f_{r})}.
 \end{equation}
 and
 \begin{equation}
   \label{eq:6}
   \vvert[\sum_{r\in R_{i,2}}f_{r}(x_{i})]\leq
   \frac{6}{m_{2j_{i}}^{2}}\# R_{i,2}
 \end{equation}
 Using that $w(f_{r})\geq  m_{2j_{1}}\geq n_{2j+1}$, by  \eqref{eq:5},\eqref{eq:6} we finish the proof.
 \end{proof}
\section{The quasi-minimality}
We shall prove the quasi-minimality in two steps. In the first step we shall handle a special case. More precisely we shall consider  block sequences $(y_{k})_{k=1}^{n_{2j+1}}$, $(z_{k})_{k=1}^{n_{2j+1}}$ such that $x_{k}=y_{k}+z_{k}$, $k\in\N$ for some dependent sequence $(x_{k},f_{k})_{k=1}^{n_{2j+1}}$. For a suitable splittings of $(x_k)$ we show that for every $f\in W$ there exists $g\in W$ such that $\frac{1}{2}f(\frac{m_{2j+1}}{n_{2j+1}}\sum_{k}y_{k})-2\frac{31}{m_{2j+1}}
\leq g(\frac{m_{2j+1}}{n_{2j+1}}\sum_{k}z_{k}))$.

In the second step we prove the quasi-minimality of $\mc{X}_4$ basing on the first step.

Let  $(x_{k},f_{k})_{k=1}^{n_{2j+1}}$ be a $(6,2j+1)$-dependent sequence such that each exact pair $(x_{k},f_{k})$ is of the form as in the proof of Proposition~\ref{exexact}.
Split each $x_k$ and $f_k$ as follows.
$$
x_k=y_k+z_k=
\frac{m_{2j_k}}{n_{2j_k}}
\sum_{i=1}^{n_{2j_{k}}/2}(y_{k,i}+z_{k,i}),\quad
f_{k}=\frac{1}{m_{2j_{k}}}\sum_{i=1}^{n_{2j_{k}}/2}(y_{k,i}^{*}+z_{k,i}^{*}),
$$
where for every $i$, $y_{k,i}<z_{k,i}<y_{k,i+1}$,
$y_{k,i}^{*}(y_{k,i})=1=z_{k,i}^{*}(z_{k,i})$ and
$\ran(y_{k,i}^{*})=\ran (y_{k,i})$,
$\ran(z_{k,i}^{*})=\ran(z_{k,i})$.

Set
$$y=\frac{m_{2j+1}}{n_{2j+1}}\sum_{k=1}^{n_{2j+1}}y_{k}=
\frac{m_{2j+1}}{n_{2j+1}}\sum_{k=1}^{n_{2j+1}}
\frac{m_{2j_{k}}}{n_{2j_{k}}}\sum_{i=1}^{n_{2j_{k}}/2}y_{k,i}
$$
and
$$
z=\frac{m_{2j+1}}{n_{2j+1}}\sum_{k=1}^{n_{2j+1}}z_{k}=
\frac{m_{2j+1}}{n_{2j+1}}\sum_{k=1}^{n_{2j+1}}
\frac{m_{2j_{k}}}{n_{2j_{k}}}\sum_{i=1}^{n_{2j_{k}}/2}z_{k,i}
$$
\begin{proposition}
  \label{qm1}
Let $y,z$ be as above. For all $f\in W$ there exists $g\in W$
such that $\vvert[f]=\vvert[g]$ and
\begin{equation}
  \label{eq:1}
 g(z)\geq \frac{1}{2}f(y)-2\frac{31}{m_{2j+1}}
\end{equation}
\end{proposition}
\begin{proof}
Let $(\fa)_{\al\in\mc{A}}$ be a tree analysis  of $f$. Set $I_{1}$ to be the set  of $k\in\{1,\dots,n_{2j+1}\}$ such that there exists $\al_{k}\in\mc{A}$ with $ w(f_{\al_{k}})=m_{2j+1}^{-1}$, $f_{\al_{k}}=E_{\al_{k}}\frac{1}{m_{2j+1}}\sum_{i=1}^{n_{2j+1}}f^{\al_{k}}_{i}$ be a special functional satisfying
\begin{enumerate}
\item[A)] $\ran(x_k)\subset  E_{\al_{k}}$.
\item[B)] $\vvert[f_{k}^{\al_{k}}]=\vvert[f_{k}]$.
\item[C)] $\prod_{\beta\prec\al_{k}}w(\fb)<m_{2j+1}$.
\end{enumerate}
We define  $I_{2}$ as $I_{1}$ with the exception that C) is replaced by
\begin{enumerate}
\item[$C_1)$]  $\prod_{\beta\prec\al_{k}}w(\fb)\geq  m_{2j+1}$.
\end{enumerate}
The complement of $I_{1}\cup I_{2}$ is denoted by $I_{3}$. Set
$$
w_{i}=\frac{m_{2j+1}}{n_{2j+1}}\sum_{k\in I_{i}}x_{k}\,\,\,\textrm{for}\,i=1,2,3.
$$
We shall estimate $f$ on $w_{1},w_{2},w_{3}$.
\begin{lemma}
The following holds
\begin{equation}
  \label{eq:9}
 \vvert[f(w_3)]<\frac{24}{m_{2j+1}}.
\end{equation}
\end{lemma}
\begin{proof}
By Proposition~\ref{eRIS} it is enough to show that for every
$\al\in\mc{A}$ with $w(\fa)=m_{2j+1}$ and every interval $E$ the
following holds
  \begin{equation}
    \label{eq:15}
\vvert[\fa(\sum_{k\in I_{3}\cap E}x_{k})]<6(1+\#E/m_{2j+1}^{2}).
  \end{equation}
  Note that  if $k\in I_{3}$, $\al\in\mc{A}$ with $\fa=m_{2j+1}^{-1}E_{\al}\sum_{k=1}^{n_{2j+1}}f_{k}^{\al}$ and
 $\ran(\fa)\cap\ran(x_{k})\ne\emptyset$
$$
\mbox{either $\vvert[f_{k}^{\al}]\ne\vvert[f_{k}]$
or $\vvert[f_{k}^{\al}]=\vvert[f_{k}]$ and
$\ran(x_{k})\nsubseteq  E_{\al}$.}
$$
Let $k_{0}=\min\{k:\ran(\fa)\cap\ran(x_{k})\ne\emptyset\}$ and $k_{1}=\min \Sigma$ where
\begin{center}
$  \displaystyle{\Sigma=\{k\leq n_{2j+1}: \vvert[f_{k}^{\al}]\ne
\vvert[f_{k}]\,\,\textrm{and}\, \ran(\fa)\cap\ran(x_{k})\neq
\emptyset\}}$
\end{center}
If $\Sigma=\emptyset$ it follows that
$\ran(\fa)\cap\ran(x_{k})\ne\emptyset$ for at most two $k$  and
hence \eqref{eq:15} holds. For every $k>k_{1}$ we have that
$w(f^{\al}_{i})\ne m_{2j_{k}}$ for every $i$.
Corollary~\ref{spRIS}  yields that
\begin{equation}\label{eq:3}
  \vvert[\fa(\sum_{k> k_{1}}x_{k})]\leq \frac{1}{m_{2j+2}^{2}m_{2j+1}}.
\end{equation}
For all $k_{0}<k<k_{1}$, $k\in (I_{1}\cup I_{2})$ except maybe for $k_{0}$. Indeed assume that some  $k_{0}<k<k_{1}$ is not in $(I_{1}\cup I_{2})$. Then by the definition of $k_{1}$ it follows that $\vvert[f_{k}^{\al}]=\vvert[f_{k}]$ and $\ran(x_{k})\nsubseteq \ran(E_{\al})$, which yields a contradiction since  $k\in I_3$ and $\ran(f_{k}^{\al})=\ran(f_{k})\subset\ran(x_{k})\subset\ran (E_{\al})$.

Thus for  $k_{0}$ (or  $k_{1})$  there exists at most one $i$ with $w(f^{\al}_{i})=m_{2j_{k_{0}}}$  (or $w(f^{\al}_{i})=m_{2j_{k_{1}}}$),  hence
\begin{equation}
  \label{eq:8}
\vvert[\fa(x_{k_{0}}+x_{k_{1}})]\leq \frac{12}{m_{2j+1}}+\frac{1}{m_{2j+1}m_{2j+2}^{2}}.
\end{equation}
To finish the proof note that \eqref{eq:3} and \eqref{eq:8} yield \eqref{eq:15}.
\end{proof}
\begin{lemma} The following holds
\begin{equation}
  \label{eq:12}
\vvert[f(w_{2})]=\vvert[f(\frac{m_{2j+1}}{n_{2j+1}}\sum_{k\in I_{2}}x_{k})]\leq\frac{6}{m_{2j+1}}.
\end{equation}
\end{lemma}
\begin{proof}
Note that for $k\in I_{2}$ it holds that $\ran(f_{k})\subset\ran(x_{k})$  and $\vvert[f_{k}^{\al_{k}}]=\vvert[f_{k}]$.
It follows
$$
\vvert[f_{\al_{k}}(\frac{m_{2j+1}}{n_{2j+1}}x_{k})]
=\frac{1}{m_{2j+1}}\frac{m_{2j+1}}{n_{2j+1}}\vvert[f_{k}(x_{k})]\leq\frac{6}{n_{2j+1}}.
$$
Property $C_{1})$ implies that
$\vvert[f(\frac{m_{2j+1}}{n_{2j+1}}x_{k})]\leq\frac{6}{n_{2j+1}m_{2j+1}}$.
Summing over $I_{2}$ we obtain \eqref{eq:12}.
\end{proof}
It remains to estimate $f$ on $w_{1}$. We shall consider a partition of $w_{1}$ into three vectors which is imposed by the  $G$-functionals.

Let $\al\in\mc{A}$ such that $\al\prec\al_{k}$ for some $k\in I_{1}$ and $\fa=\frac{1}{2}S\fb$ is $G$-functional determined  by the intervals $q_{\al}\leq n_{1}^{\al}<\dots<n_{q_{\al}}^{\al}$.  We set
\begin{align*}
L_{1}^{\al}&=\{(k,i):k\in I_{1}, \al\prec\al_{k}\,\textrm{and exists}\,
d\leq q_{\al}\,\textrm{with}\, n_{d}^{\al}\in\ran(y_{k,i}+z_{k,i})\,\textrm{for some}\, i\},
\\
L_{2}^{\al}&=\{(k,i):k\in I_{1}, \al\prec\al_{k}\,\textrm{and
exists}\, d\leq
q_{\al}/2\,\textrm{with}\,\ran(y_{k,i}+z_{k,i})\subset(n_{2d}^{\al},n_{2d+1}^{\al})\},
\\
L_3^{\al}&=\{(k,i): k\in I_{1},\al\prec\al_{k}\,\,\textrm{and}\,\,
(k,i)\not\in L_{1}^{\al}\cup L_{2}^{\al}\}.
\end{align*}
We also set
$$\Gamma=\{\al\in\mc{A}:\al\prec\al_{k}\,\textrm{for some $k\in I_{1}$ and $\fa$ is $G$-functional}\},
$$
and $L_{i}=\cup_{\al\in \Gamma}L_{i}^{\al}$ for $i=1,2,3$. Without loss of generality we can assume that these sets define a partition of the whole vector $w_1$.
\begin{remark}\label{u3}
An easy inductive argument yields that for every $(k,i)\in L_{3}$
and every $\al\prec\al_{k}$ we have that
$\supp(y_{k,i}+z_{k,i})\cap\supp(\fa)=\supp(y_{k,i}+z_{k,i})\cap\supp(f_{\al_{k}})$.

Indeed, for a $G$-functional $\fa$ the above follows from the
definition of the set $L_{3}^{\al}$, as $\fa$ has one successor
$\fb$ and satisfies
$\fa(y_{k,i}+z_{k,i})=\frac{1}{2}\fb(x_{k,i}+z_{k,i})$.

If $\fa$ is weighted functional then there exists a unique $\be\in S_{\al}$ such that $\supp(y_{k,i}+z_{k,i})\subset\supp (\fb)$ and $\fa(y_{k,i}+z_{k,i})=\frac{1}{w(\fa)}\fb(x_{k,i}+z_{k,i})$.
\end{remark}
The following lemma give us an upper bound for the cardinality of the
set $\Gamma$.
\begin{lemma}\label{l1}
  \label{slenght}
Let $(\fa)_{\al\in\mc{A}}$ be a tree analysis of a functional. Let
$$
B=\{\al\in\mc{A}:\prod_{\beta\precneqq\alpha}w(\fb)<m_{2j+1}\}
$$
Then $\# B\leq (5n_{2j})^{\log_{2}(m_{2j+1})-1} \leq n_{2j+1}^{1/3}$.
\end{lemma}
For the proof  we refer to the proof of  Lemma II.9, \cite{ATod}.

The sets $L_{1}^{\al}, L_{2}^{\al}$ and  $L_3^{\al}$   implies the following partition of $w_{1}$.
$$
u_{i}=\frac{m_{2j+1}}{n_{2j+1}}\sum_{k\in I_{1}} \frac{m_{2j_{k}}}{n_{2j_{k}}}\sum_{(k,i)\in L_{i}}(y_{k,i}+z_{k,i}),\,\,i=1,2,3.
$$
\begin{lemma}
  \label{u2}
We have that $f(u_{2})=0$.
\end{lemma}
\begin{proof}
  Let $(k,i)\in L_{2}^{\al}$ for some $\al$. From the definition of  $L_{2}^{\al}$ it follows that $\supp (\fb)\cap \supp (x_{k,i})=\emptyset$ for every $\be\preceq\alpha$. It follows that $f(u_{2})=0$.
\end{proof}
\begin{lemma}
  It holds that
  \begin{equation}
 \label{eq:10}
\vvert[f(u_{1})]=\vvert[f(\frac{m_{2j+1}}{n_{2j+1}}
  \sum_{k\in I_{1}}\frac{m_{2j_{k}}}{n_{2j_{k}}}\sum_{i:(k,i)\in L_{1}}(y_{k,i}+z_{k,i}))]\leq
\frac{1}{m_{2j+1}^{2}}.
\end{equation}
\end{lemma}
\begin{proof}
Let $\fa$ be a $G$-functional determined by the intervals $q_{\al}\leq n_{1}^{\al}<\dots<n_{q_{\al}}^{\al}$. Let $k_{0}$ be the smallest  $k$ such that $(k,i)\in L_{1}^{\al}$ for some $i\leq n_{2j_{k}}$. It follows that $q_{\al}\leq\maxsupp x_{k_{0}}$. If $k_{0}<n_{2j+1}$, as $(x_{k},x_{k}^{*})_{k}$ is a dependent sequence, we have that $q_{\al}\leq\maxsupp x_{k_{0}}\leq m_{2j_{k_{0}+1}}$. Therefore
$$
\# \{ (k,i)\in L_{1}^{\al}:k>k_{0} \}\leq m_{2j_{k_{0}+1}}.
$$
The above inequality implies the following.
\begin{align*}
\vvert[\fa(\sum_{(k,i)\in L_{1}^{\al}}\frac{m_{2j_{k}}}{n_{2j_{k}}}(y_{k,i}+z_{k,i}))]
 &=
\sum_{k\in I_{1}}
\vvert[\fa(\frac{m_{2j_{k}}}{n_{2j_{k}}}\sum_{i:(k,i)\in L_{1}^{\al} }(y_{k,i}+z_{k,i}))]
 \\
 &\leq \norm[x_{k_{0}}]+\sum_{k>k_{0}}6m_{2j_{k_{0}+1}}\frac{m_{2j_{k}}}{n_{2j_{k}}}\leq
6+\frac{1}{n_{2j+2}}\leq 7.
\end{align*}
Since for every $k\in I_{1}$ we have $\prod_{\be\prec\al_{k}}w(\fb)<m_{2j+1}$,  Lemma~\ref{l1} yields that  $\#\Gamma\leq (5n_{2j})^{\log_{2}(m_{2j+1})-1}\leq n_{2j+1}^{1/3}$. Therefore
\begin{align*}
\vvert[f(u_{1})]\leq
    &
\frac{m_{2j+1}}{n_{2j+1}} \sum_{\al\in \Gamma}
\vvert[\fa(\sum_{(k,i)\in L_{1}^{\al}}\frac{m_{2j_{k}}}{n_{2j_{k}}}(y_{k,i}+z_{k,i}))]
 \\
 &\leq \frac{m_{2j+1}}{n_{2j+1}}7n_{2j+1}^{1/3}\leq\frac{1}{m^{2}_{2j+1}}.
 \end{align*}
\end{proof}
It remains to estimate the action of $f$ on $u_{3}$.
\begin{lemma}
  \label{u3-lemma}
 Let $y_3=u_3|_{\supp y}$ and $z_{3}=u_3|_{\supp z}$. There exist a functional $g\in W$ with $\vvert[g]=\vvert[f]$ satisfying
 \begin{equation}
   \label{eq:17}
g(z_{3})\geq \frac{1}{2}f(y_{3}).
 \end{equation}
\end{lemma}
Recall that  for  every $k\in I_{1}$ it holds that $\vvert[f_{k}^{\al_{k}}]=\vvert[f_{k}]$ and $\ran(x_{k})\subset E_{\al_{k}}$. Also since $\prod_{\beta\prec\al_{k}}w(\fb)<m_{2j+1}$  it follows that the nodes $\al_{k},\al_{l}$ are incomparable for $k\ne l\in I_{1}$.

Let $g$ be the functional defined by a tree-analysis we obtain by replacing each $f_{k}^{\al_{k}}$ by $f_{k}$ for every $k\in I_{1}$.  Lemma ~\ref{newtree}  yields that the resulting functional is a norming functional.

Setting $y_{3,k}=u_{3}|_{\supp y_k}$ and $z_{3,k}=u_3|_{\supp z_k}$ and using that $\vvert[f_{k}^{\al_{k}}]=\vvert[f_{k}]$ we have the following
\begin{equation}
  \label{eq:16}
 \vvert[f_{k}^{\al_{k}}(y_{3,k})]\leq 2f_{k}(z_{3,k})
\end{equation}
Remark~\ref{u3} yields that for every $\gamma\prec \al_{k}$ we have that
$$
 \vvert[f_{\gamma}(y_{k,3})]\leq
\left(\prod_{\gamma\preceq\delta\preceq \al_{k}}w(f_{\delta})^{-1}\right)\vvert[f_{k}^{\al_{k}}(y_{k,3})].
$$
Lemma~\ref{newtree} yields also that ${\displaystyle g_{\gamma}(z_{3,k})=(\prod_{\gamma\preceq\delta\preceq \al_{k}}w(f_{\delta})^{-1})f_{k}(z_{3,k}) }$ for every $\gamma\prec\al_{k}$. Therefore
\begin{equation}
  \label{eq:2}
  \vvert[f_{\gamma}(y_{3,k})]\leq
\left(\prod_{\gamma\preceq\delta\preceq \al_{k}}w(f_{\delta})^{-1}\right)\vvert[f_{k}^{\al_{k}}(y_{3,k})]\leq 2g_{\gamma}(z_{3,k}).
\end{equation}
The above inequality proves \eqref{eq:17}.
\end{proof}
Combining \eqref{eq:9}, \eqref{eq:12}, \eqref{eq:10} and \eqref{eq:17} we obtain that
$$
g(z)\geq \frac{1}{2}f(y)-2\frac{31}{m_{2j+1}}.
$$
\begin{theorem}
 The space $\mc{X}_{(4)}$ is quasi-minimal.
\end{theorem}

\begin{proof}

The proof is based  on the arguments we use  in the proof of Proposition~\ref{qm1}.

Let $Y,Z$ be two block subspaces of $\mc{X}_{(4)}$. Inductively, by Proposition \ref{exexact}, we choose a  a sequence $(x_{l},x^{*}_{l})_{l\in\N}$ such that  $(x_{l},x_{l}^{*})$ is a $(2j_{l}+1)$-dependent sequence, $3<j_{l}\nearrow +\infty$, which splits as in the first step i.e.
$$
x_{l}=y_{l}+z_{l}=\frac{m_{2j_{l}+1}}{n_{2j_{l}+1}} \sum_{k=1}^{n_{2j_{l}+1}/2}y_{l,k}+z_{l,k}
$$
where
${\displaystyle y_{l,k}+z_{l,k}=\frac{m_{2j_{k,l}}}{n_{2j_{l,k}}} (y_{l,k,1}+z_{l,k,1}+\dots+y_{l,k,n_{2j_{l,k}}}+z_{l,k,n_{2j_{l,k}}}) }$
and additionally  $y_{l,k,i}\in Y$, $z_{l,k,i}\in Z$.

We may also assume that the weights that appear in the choice of
the dependent sequence $(x_{l},x_{l}^{*})$ are bigger than the
weights we use in $(x_{l-1},x_{l-1}^{*})$.

Let $\norm[\sum_{l}a_{l}y_{l}]=1$ and let $f\in W$ with $f(\sum_{l}a_{l}y_{l})>1/2$. Let $(\fa)_{\al\in\mc{A}}$ be a tree-analysis  of $f$.

We define for every $l\in\N$ the set  $I_{l,1}, I_{l,2}, I_{l,3}$ as in Proposition~\ref{qm1}.  For each of the sets $I_{l,1}$, $I_{l,2}$ we get
$$
f(\sum_{k\in I_{l,3}}(y_{l,k}+z_{l,k}))\leq\frac{24}{m_{2j_{l}+1}}
$$
and
$$
f(\sum_{k\in I_{l,2}}(y_{l,k}+z_{l,k}))\leq\frac{6}{m_{2j_{l}+1}}
$$
For the sets $I_{l,1}$ as in Proposition~\ref{qm1} we define the sets $L_{l,1}, L_{l,2}$ and $L_{l,3}$ and vectors $u_{l,1}, u_{l,2}, u_{l,3}$. As before we obtain that $f( u_{l,2})=0$ and $f(u_{l,1})$ is dominated by $m_{2j_{l}+1}^{-1}$.

For the sets $L_{l,1}$ we  work as in Lemma~\ref{u3-lemma} substituting $f_{l,k}^{\al_{k}}$ (which corresponds to $f_{k}^{\al_{k}}$ of Lemma~\ref{u3-lemma}) by the functional $f_{l,k}$ (which corresponds to $f_{k}$ of Lemma~\ref{u3-lemma}).   In this way we get a functional $g$ such that $\vvert[g]=\vvert[f]$ and
$$
g(\sum_{l}a_{l}u_{l,3}{|}_{\supp z_l})\geq\frac{1}{2} f(\sum_{l}a_{l}u_{l,3}{|}_{\supp z_l}).
$$
The above yields that
$$
g(\sum_{l}a_{l}z_{l})\geq
\frac{1}{2}f(\sum_{l}a_{l}y_{l})- 2\sum_{l}\vvert[a_{l}] \frac{31}{m_{2j_{l}+1}}\geq \frac{1}{5}.
$$
which ends the proof. \end{proof}
\section{Tightness by range}
We show now, using $G$-operations, the following
\begin{theorem}
 The space $\mc{X}_4$ is tight by range.
\end{theorem}
\begin{proof} Let $(x_{i})$ be normalized block sequence. We show that there  exists  no bounded operator  $T$ such that  $\supp T(x_{i})\cap\ran(x_{i})=\emptyset$ and  $T$ can be extended  to an isomorphism from $[(x_{i})]$  to $X$. This will prove that $\mc{X}_{(4)}$ is tight by range.

Let $T$ be an operator as above and assume without loss of generality that $\norm[T]\leq 1$. By the reflexivity of the space and passing to a subsequence we may assume that $(T(x_i))_{i}$  is a block sequence and moreover
$$
\ran(x_{i}+Tx_{i})<\ran(x_{i+1}+Tx_{i+1})\,\, \forall i\in\N.
$$
Let  $j\in\N$.  By Lemma~\ref{esaverage} we can  choose $w=\sum_{k=1}^{l(n_{2j})}a_{k}x_{r(j)+k}$ such that  $w$  is an $2-\ell^{n_{2j}}$-average and $l(n_{2j})\leq w$.

Let $f_{0}\in W$ be a functional such that $f_{0}(w)\geq 1$. Without loss of generality we may assume that $\ran(f_{0})=\ran(w)$. For  $k\leq l(n_{2j})/2$ set
$$
E_{2k-1}=\ran(x_{r(j)+2k-1})\,\,\textrm{and}\,\,E_{2k}=(\maxran x_{r(j)+2k-1}, \minsupp x_{r(j)+2k+1}),
$$
It follows that $(E_{k})_{k=1}^{l(n_{2j})}$ are consecutive intervals  and $l(n_{2j})\leq E_{1}<\dots<E_{n_{2j}}$.
Then the functional $f=\frac{1}{2}\sum_{k=1}^{n_{2j}/2}E_{2k-1}f_{0}\in W$ and without loss of generality we may assume $f(w)\geq 1/2$. Otherwise we can take the restriction of $f$ to $\cup_{k} \ran(x_{r(j)+2k})$. It follows that $f$  satisfies
\begin{equation}
  \label{eq:11}
f(w)\geq \frac{1}{2}\,\,\textrm{and}\,\,\supp(Tx_{i})\cap \supp (f)=\emptyset\,\,\forall  i\in\N.
\end{equation}
\begin{definition}
We shall call the pair $(w,f)$ a $(2,n_{2j})$-special pair disjoint from $(Tx_{i})$ if $w$ is an $2-\ell_{1}^{n_{2j}}$-average and \eqref{eq:11} holds.
\end{definition}
Let $j\in\N$.  Inductively, by Proposition~\ref{exexact}, we choose a $(6,n_{2j+1})$-dependent sequence $((u_{i},f_{i}))_{i\leq n_{2j+1}}$ such that
\begin{enumerate}
\item[P1)] $(u_{i},f_{i})$ is a $(6,\e_{i})$-exact pair of the form 
  $u_{i}=\frac{m_{2j_{i}}}{n_{2j_{i}}}\sum_{k=1}^{n_{2j_{i}}}u_{i,k}$ and
$f_{i}=\frac{1}{m_{2j_{i}}}\sum_{k=1}^{n_{2j_{i}}}f_{i,k}$ for any $i$,
\item[P2)]  $(u_{i,k},f_{i,k})$  is  a $(2,n_{2j_{i,k}})$-special pair disjoint from $(Tx_{i})_{i}$ for every $i,k$,
 \item[P3)] $(f_{1},\dots,f_{n_{2j+1}})$ is a $(2j+1)$-special sequence.
\end{enumerate}
Note that
\begin{equation}
  \label{eq:21}
\norm[\frac{m_{2j+1}}{n_{2j+1}}\sum_{i=1}^{n_{2j+1}}u_{i}]\geq  \frac{1}{m_{2j+1}}
\sum_{i=1}^{n_{2j+1}}f_{i}(\frac{m_{2j+1}}{n_{2j+1}}\sum_{i=1}^{n_{2j+1}}u_{i})\geq \frac{1}{2}.
\end{equation}
Moreover $\supp(f)\cap\supp (Tx_j)=\emptyset$ for all $f=\frac{1}{m_{2j+1}}E\sum_{i=1}^{n_{2j+1}}h_{i}$ and $\vvert[h_{i}]=\vvert[f_{i}]$ for all $i$.

\begin{lemma}\label{llast}
  The following holds
  \begin{equation}
    \label{eq:18}
\norm[\frac{1}{n_{2j+1}}\sum_{i=1}^{n_{2j+1}}Tu_{i}]\leq 25m_{2j+1}^{-2}
  \end{equation}
\end{lemma}

\begin{proof}[Proof of Lemma~\ref{llast}]
Let $u_{i,k}=\sum_{j\in D_{i,k}}a_{j}x_j=
\frac{1}{n_{2j_{i,k}}}(\bar{x}_{i,k,1}+\dots\bar{x}_{i,k,n_{2j_{i,k}}})$
be a $2-\ell_{1}^{n_{2j_{i,k}}}$ average. We set
$y_{i,k}=\sum_{j\in D^{*}_{i,k}}a_{j}Tx_{j}$ where
$D^{*}_{i,k}=D_{i,k}\setminus\{\max D_{i,k}\}$.

Note that $\maxsupp y_{i,k}<\maxsupp u_{i,k}$,
$$
y_{i,k}=\frac{1}{n_{2j_{i,k}}} (\bar{y}_{i,k,1}+\dots+\bar{y}_{i,k,n_{2j_{i,k}}})
$$
and $\norm[\bar{y}_{i,k,j}]\leq \norm[T]\norm[\bar{x}_{i,k,j}]\leq
2$ for all $j$.

Set $w_{i}=\frac{m_{2j_{i}}}{n_{2j_{i}}}\sum_{k=1}^{n_{2j_{i}}}y_{i,k}$. Since
$$
\norm[Tu_{i}-w_{i}]\leq
\frac{m_{2j_{i}}}{n_{2j_{i}}}\sum_{k=1}^{n_{2j_{i}}} \norm[Tu_{i,k}-y_{i,k}]\leq
\frac{m_{2j_{i}}}{n_{2j_{i}}}\sum_{k=1}^{n_{2j_{i}}} \frac{1}{n_{2j_{i,k}}}\leq\frac{1}{m^{2}_{2j_{i}}}
 $$
 it is enough to show that
$$
\norm[\frac{1}{n_{2j+1}}\sum_{i=1}^{n_{2j+1}}w_{i}]\leq 24m_{2j+1}^{-2}
$$
To get the above inequality we shall use Proposition~\ref{eRIS}. We show that $(y_{i,k})_{k=1}^{n_{2j_{i}}}$ is a $(3,\e_{i})$-RIS.

Indeed as in Lemma~\ref{l1average} we obtain that for all $f\in W$ with $w(f)=m_{p}<m_{2j_{i,k}}$ it holds that
$$
\vvert[f(y_{i,k})]\leq\frac{3}{m_{p}}.
$$
Also since $m_{2j_{i,k+1}}^{-1}\max\supp u_{i,k}\leq \e_{i} $ and $\maxsupp y_{i,k}<\maxsupp u_{i,k}$ we get that $(y_{i,k})_{k\leq  n_{2j_{i}}}$ is a $(3,\e_{i})$-RIS. By Proposition~\ref{eRIS}  and P3) we have that $(w_{i})_{i=1}^{n_{2j+1}}$ is a $(6,n^{-1}_{2j+1})$-RIS.

 We will  show now that for every $f\in W$ with $w(f)=m_{2j+1}^{-1}$ and every interval $E$ we have that
 $$
\vvert[f(\sum_{i\in E}w_{i})] \leq 6(1+ \frac{\# E}{m_{2j+1}^{2}})
 $$
Let $f=E\frac{1}{m_{2j+1}}\sum_{r=1}^{n_{2j+1}}h_r$ be a special functional. Let $i_0=\min\{ i: \vvert[h_{i}]\ne \vvert[f_{i}] \}$. If $i_{0}>1$ it follows that for every $i<i_{0}$ $f(w_{i})=0$,  since by P2) $f(Tu_{i,k})=0$ for every $(i,k)$ with $i<i_{0}$. For every $i>i_0$ the assumptions of Corollary~\ref{spRIS} hold,  hence
\begin{equation}
  \label{eq:13}
  \vvert[f(\sum_{i>i_{0}}w_{i})]\leq\frac{1}{m_{2j+2}^{2}}
\end{equation}
For the $w_{i_0}$, $w(h_{i_{0}})=w(f_{i_{0}})$ and  hence using Corollary~\ref{spRIS} we get
\begin{equation}
  \label{eq:14}
 \vvert[f(w_{i_{0}})]\leq 6+\frac{1}{m_{2j+2}^{2}}
\end{equation}
It follows
$$
\vvert[f(\sum_{i\in E}w_{i})] \leq 6(1+ \frac{\# E}{m_{2j+1}^{2}})
$$
Proposition~\ref{eRIS} yields
\begin{equation*}
  \norm[\frac{m_{2j+1}}{n_{2j+1}}\sum_{i=1}^{n_{2j+1}}w_i]\leq\frac{24}{m_{2j+1}},
\end{equation*}
which ends the proof of the lemma.
\end{proof}
Notice that combining \eqref{eq:21} and \eqref{eq:18} we get that $T$ is not an isomorphism, which ends the proof of the theorem.
\end{proof}

\end{document}